\documentclass[12pt]{article}

\title{Speed exponents of random walks on groups}
\date{\today}
\author{Gideon Amir  \and B\'alint Vir\'ag}

\usepackage{amsmath,amsthm,amssymb,amsfonts,
enumerate,graphicx}

\usepackage[longnamesfirst]{natbib}
\bibpunct[ ]{(}{)}{,}{a}{}{,}

\theoremstyle{plain}
    \newtheorem{theorem}{Theorem}
    
    \newtheorem{lemma}[theorem]{Lemma}
    \newtheorem{proposition}[theorem]{Proposition}
    \newtheorem{corollary}[theorem]{Corollary}
    
    \newtheorem{conjecture}[theorem]{Conjecture}
    \newtheorem{problem}[theorem]{Problem}
    \newtheorem{question}[theorem]{Question}

\theoremstyle{definition} 
    
    \newtheorem{remark}[theorem]{Remark}

\newcommand{\Z}{{\mathbb Z}}

\newcommand{\E}{{\mathbb E}}

\newcommand{\pr}{{\mathbf P}}
\newcommand{\R}{{\mathbb R}}
\newcommand{\one}{{\mathbf 1}}

\newcommand{\tree }{\mathbb T}

\newcommand{\Aut}{\operatorname{Aut}}

\newcommand{\Ent}{H}

\renewcommand{\bar}{\overline}

\newcommand{\ev}{\mbox{\bf E}}

\newcommand{\res}{\mathfrak R}

\newcommand{\mm}{{\mathcal M_{m}}}

\newcommand{\tm}{{\tree_{m}}}
\newcommand{\ttm}{{\tree_{\tau m}}}

\DeclareMathOperator{\supp}{supp}

\begin{document}
\maketitle
\begin{abstract}
For every $3/4\le \beta< 1$ we construct a finitely generated group so that the expected distance of the simple random walk from its starting point satisfies  $\ev |X_n|\asymp n^{\beta}$. In fact, the speed can be set precisely to equal any nice prescribed function up to a constant factor.
\end{abstract}

\section{Introduction}\label{s.intro}

The central question in the theory of random walks on groups is how symmetries of the underlying space give rise to structure and rigidity of the random walks. For example, for nilpotent groups, it follows easily from bounds in \cite{HSc} that random walks have diffusive behavior, namely that the rate of escape, defined as the expected distance of the walk from the identity satisfies \footnote{The notation $f\asymp g $ denotes that there exist constants $c,C$ independent of $n$ such that $cf(n)\leq g\leq Cf(n)$.}  $$\ev |X_n| \asymp n^{1/2}.$$  (See also \cite{Thompson} for more classes with this behavior.) On nonamenable groups, on the other hand, we have
$$\ev |X_n| \asymp n.$$
Starting with the works of \cite{Kesten59}, the rate of escape has been connected to many other properties of the group, including the spectral radius, the volume growth, rate of entropy growth, and the Liouville property. More recently, \cite{ANP} and \cite{NaorPeres09} have established a beautiful connection between the rate of escape and embeddings of groups into Hilbert space, see \eqref{e:ANP}.

Their work renewed interest in a long-standing question, attributed to Vershik, namely, what rates of escape are possible. In the formulation of Peres (personal communication):
$$
\mbox{For what $\beta$ is there a random walk on a group with }\ev |X_n| = n^{\beta+o(1)}?
$$
The only known answers to this question were given by \cite{ErschlerDrift}, who realized that for a random walk on a group $G$ satisfying $\ev |X_n| \asymp n^\beta$, the wreath product $\mathbb Z\wr G$ has $\ev|X_n| \asymp n^{\frac{\beta+1}2}$.

Thus Erschler showed that the canonical random walk on the $k-1$-times iterated wreath product of $\mathbb Z$ satisfies
$$\ev |X_n|\asymp n^{1-2^{-k}}.$$

The main goal of this paper is to show that in fact, all speed exponents between $3/4$ and $1$ are achievable in residually finite groups. In fact, we show much more: our groups allow us to adjust precisely $\ev |X_n|\asymp f(n)$ for a general class of functions $f$. For this, we use the permutation wreath product of $\mathbb Z$ and the action $S$ of the so called piecewise mother groups $\mm$  on the boundary of their tree; see Section \ref{s:p mother groups} for definitions.

\begin{theorem}\label{t:main speed}
For any $\gamma\in[3/4,1)$ and any function $f:\R_+\rightarrow \R_+$ satisfying
$f(1)=1$ and the log-Lipshitz condition that for all real $a,n\ge 1$
$$a^{3/4}f(n)\leq f(an)\leq a^\gamma f(n)$$
 there is
a bounded sequence $(m_\ell)$ such the random walk on $\mathbb  Z \wr_S \mm $ satisfies that with constants depending on $\gamma$ only,
$
\ev|X_n|  \asymp f(n).
$
\end{theorem}
For the theorem and throughout this paper, all random walks are symmetric and the step distribution has finite support.
While the actual group construction for Theorem \ref{t:main speed} is simple, the theory behind the analysis is more sophisticated and has many potential applications. It is based on the analysis of permutation wreath products (see Section \ref{s.speed in pw} for definitions).

\begin{theorem}\label{t:maingeneral}
Let $Y_n$ be a random walk on a group $\Gamma$ which acts on a set $S$. Fix $o\in S$,and let $p_n$ be the probability that the random walk $o.Y_n$ on $S$ does not return to $o$ in the first $n$ steps.
Assume that there is a constant $c$ so that for all  $n\ge 1$ we have
\begin{equation}\label{e:cond1}
\ev |Y_n| + H[\{o, o.Y_1^{-1},\ldots, o.Y_n^{-1}\}]+(p_0+\ldots +p_n) \le c np_n,
\end{equation}
where the $H$ term is the entropy of a random set called the \textbf{inverted orbit}.
Let $\Lambda$ be an infinite group and let $\lambda(n)$ be the expected displacement for a generating random walk in $\Lambda$. Then the switch-walk-switch random walk $X_n$ on the permutation wreath product $\Lambda \wr_S \Gamma$ satisfies
$$
\ev |X_n| \asymp np_n \lambda(1/p_n).
$$
\end{theorem}

The main reason this theorem is useful is that it gives control of the expected displacement $\ev |X_n|$, usually a difficult quantity, in terms of the function $p_n$, which is often easier to compute. The condition \eqref{e:cond1} ensures that the action is significant enough to determine the speed of the switch-walk-switch random walk.

 For a simple example,
consider $\Gamma = \mathbb Z$, with the free (Cayley) action. Then we are back to the ordinary wreath product, $p_n\asymp n^{-1/2}$, and \ev $|Y_n|\asymp n^{1/2}$, and the entropy in \eqref{e:cond1} is of order $\log n$. Repeated application of Theorem \ref{t:maingeneral} then recovers the result of \cite{ErschlerDrift}.

Because of Theorem \ref{t:maingeneral}, understanding speed and entropy in wreath products can boil down to understanding return times of random walks on Schreier graphs, a much easier task.  Finally, we make this easier by constructing a family of Schreier graphs on which the random walk simply projects to nonnegative birth-and-death Markov chains. These are chains on the nonnegative integers where only nearest-neighbor moves allowed; the random walk there can be analyzed many ways, including electrical network theory.
The corresponding action satisfies the condition \eqref{e:cond1}, and Theorem \ref{t:main speed} holds since $p_n$ can be adjusted sufficiently precisely (in fact to any function as specified in Theorem \ref{t:main speed} with $3/4$ replaced by $1/2$.) Then by Theorem \ref{t:maingeneral}, the $\ev|X_n|$ can be determined up to a constant from the expected displacement in  $\Lambda$ for any infinite group $\Lambda$.

Our analysis is precise enough to give a quantitative improvement on Brieussel's theorem on entropy. In \cite{BrieusselEntropy} it is shown that entropy exponents between $1/2$ and $1$ are realizable; the same analysis can be extended to $H(X_n)=f(n)n^{o(1)}$ for a general class of functions $f$. Here we show that in fact the random walk entropy can also be adjusted much more precisely, up to a constant factor.
\begin{theorem}\label{t:main entropy}
For any $\gamma\in[1/2,1)$ and any function $f:\R_+\rightarrow \R_+$ satisfying
$f(1)=1$ and the log-Lipshitz condition that  for all real $a,n\ge 1$
$$a^{1/2}f(n)\leq f(an)\leq a^\gamma f(n)$$
 there is
a bounded sequence $(m_\ell)$ such the random walk on $\mathbb Z_2 \wr_S \mm $ satisfies
$
H(X_n) \asymp f(n)$, and the implied constants depend on $\gamma$ only.
\end{theorem}
The piecewise mother groups $\mm$ that we are using are related to those in the paper of Brieussel; the main difference is that our analysis is based on a careful geometric understanding of the birth-and-death chain structure of the underlying Markov chains; also, it is key in our analysis of the speed to let $\Lambda$ be an infinite group.

The groups $\mm$ are versions of automation groups. The most famous of such groups is the celebrated group of \cite{Grigorchuk84}. The rate of escape on automaton groups was first studied by \cite{BV05}, who used its sublinearity to show that the Basilica group is amenable. These results were extended to so-called degree zero automaton groups by \cite{BKN}, and for degree one groups by \cite{AAV}. In \cite{AV}, it was shown that degree three groups typically have linear rate of escape.

The idea to use permutation wreath product over these groups is due to\\ \cite{BartholdiErschler}; they used the Grigorchuk group as a base group to determine the exponent of the stretched exponential growth of its permutation wreath product; to do so for any intermediate growth group had been an open question. \cite{BrieusselGrowth} used a piecewise version of this construction (the likeness of which was used earlier by \cite{ErschlerFollner} to give examples for Follner functions) to show that all exponents within a range are possible. Later, \cite{PakKassabov} gave a different construction. Most recently, \cite{BrieusselEntropy} used the same groups we use to show that entropy exponents can freely vary between $1/2$ and $1$.

 The permutation wreath product can be thought of as a relative of the usual wreath product $\Lambda \wr \Gamma$, also known as the lamplighter product of $\Lambda$ and $\gamma$.  Thus we borrow terminology and think of $\Lambda$ as the "lamp group" and of the switch-walk-switch walk as a type of Lamplighter walk. This analogy is limit    ed though and must be used with care. For example, perhaps surprisingly, the lamplighter corresponding to a permutation wreath product does not light lamps along the random walk  path $o.Y_n$ on the corresponding Schreier graph with root $o$. Instead, it lights the lamps along the inverted orbit $o.Y_n^{-1}$, which may have a very different distribution (see Section \ref{s.speed in pw}).

 We give a sketch of the proof of Theorem \ref{t:main speed} below. The rest of the paper is built as follows: In Section \ref{s.inverted}, we analyze the distribution of the size of the inverted orbit. In Section \ref{s.entropy and speed} we give a precise version of the well-known inequality connecting random walk entropy and speed based on the bounds by Varopoulos and Carne. In Section \ref{s.speed in pw}, we define permutation wreath products and develop our bounds for speed and entropy in general permutation wreath products. In Section \ref{s:assembly}, we introduce and analyze the assembly line, a birth-and-death Markov chain which will be the basis of our construction; these are simpler to understand than the underlying groups.  In Section \ref{s:p mother groups} we introduce the underlying groups, prove specific estimates for the orbits in these groups and then apply the bounds of the previous sections to deduce the three main theorems. Finally, we close the paper with some open questions.

\subsection{Overview of the proof of Theorem \ref{t:main speed}}
Let us describe an overview of the construction used to attain groups with prescribed behavior of the speed as in Theorem \ref{t:main speed}.
Let $G$ be some group acting on a rooted set $(S,o)$ to be described later. Let $\Gamma= \mathbb Z \wr_S G$ be the permutation wreath product, and let $X_n=(L_n,Y_n)$ be a switch-walk-switch random walk on $\Gamma$, where $L_n$ denotes the lamp configuration at time $n$ and $Y_n$ denotes the projection of $X_n$ to the base group $G$. Let $Q_n$ be the inverted orbit, that is $Q_n=\{o,o.Y_1^{-1},\ldots,o.Y_n^{-1}\}$ and let $q_n$ denote the expected size of $Q_n$. We will divide the construction and the analysis into two main parts.
First, we will analyze $\E|X_n|$ in terms of $q_n$, and show that $\E|X_n|\asymp \sqrt{nq_n}$ and then we will construct a group and an action so that $\sqrt{n q_n}\asymp f(n)$.

To get a lower bound on $\E|X_n|$ we argue that to get from $X_n$ to the identity, one must switch off all the lights. That is, $\E|X_n|\geq \sum_{s\in\supp L_n}|L_n(s)|$.
For the sake of this overview, we will work with expectations and pretend that the walk on $S$ is evenly spread on its range. That is we pretend that $|\supp L_n|\asymp q_n$ and that each point in the inverted orbit is visited about $n/q_n$ times. As the lamps we have are $\Z$-valued, in each position in $\supp L_n$ we will reach a distance of about $\sqrt{n/q_n}$. Thus
$$\E|X_n|\geq \sum_{s\in\supp L_n}|L_n(s)| \asymp q_n \sqrt{n/q_n} \asymp \sqrt{nq_n}.$$
This argument is made precise in Theorem \ref{t:speed_lb} of Section \ref{s.speed in pw}. 

To get the upper bound, it is not clear a priori that counting lamp switches is enough (for instance, one needs to move between the different points where the lamps are lit).
Unfortunately, there are not many tools to get an upper bound on the speed of a random walk. We bound the speed by looking at the entropy of the random walk (see Section \ref{s.entropy and speed}) and using a well-known inequality relating speed and entropy (Proposition \ref{p:VC}) which roughly says that $\E|X_n| \preceq \sqrt{nH(X_n)}$. To bound the entropy $H(X_n)$ we divide the information needed to describe $X_n$ into three parts: $Y_n$, the support of $L_n$ and the value of $L_n(s)$ in each point of the support, thus we get
$$H(X_n)=H(L_n,Y_n)\leq H(L_n)+H(Y_n) \leq H(L_n \,| \, \supp L_n) + H(\supp L_n) + H(Y_n)$$

To get an upper bound on $H(L_n \,| \, \supp L_n)$ note that $|L_n(s)|\leq 2n$ for each $s\in \supp L_n$, so it only takes order $\log n$ bits to describe each entry of $L_n$, and since $\E|\supp L_n| \asymp q_n$ we get $H(L_n \,| \, \supp L_n)\leq q_n \log n$.
Controlling $H(\supp L_n)$ and $H(Y_n)$ demands precise knowledge on the group $G$ and its action. This is done in Section \ref{s:p mother groups}, but for the sake of this overview we will assume that they are both small enough so that $H(X_n)\asymp H(L_n \,| \, \supp L_n)$ (this is the condition in Theorem \ref{t:maingeneral} saying the action is significant enough).
So overall we have $H(X_n) \preceq q_n\log n$ which using the speed-entropy relation translates to $$\E|X_n|\leq \sqrt{nq_n\log n},$$ which is the same as the lower bound we got up to  a $\sqrt{\log n}$ factor. See Theorem \ref{t:Ent_upper_general} of Section \ref{s.speed in pw} for more details. Some further efforts, done in Lemma \ref{l:Gady} and Theorem \ref{t:upper_tight} lets us get rid of the $\sqrt{\log n}$ factor.\newline
It now remains to construct a group and an action for which one can control $|Q_n|$ sufficiently well. As mentioned in the introduction, and formally done in Section \ref{s.inverted}, this can be reduced to understanding return probabilities of the random walk $o.Y_n$ on the Schreier graph. The action itself is constructed and analyzed in Section \ref{s:assembly}. The set $S$ is the set of all eventually 0 infinite sequences $\{a_i\}_i\geq 1$ with $0\leq a_i\leq m_i-1$ for some bounded sequence $m=\{m_i\}_{i\geq 1}$. 

A central observation in this paper is that the analysis of the random walk on this graph boils down to the analysis of a birth and death chains.  The control over the return probabilities and thus $q_n$ and $\E|X_n|$ comes through the choice of the sequence $m$. The groups themselves are defined in Section \ref{s:p mother groups} where the above mentioned estimates for $H(\supp L_n)$ and $H(Y_n)$ are proved as well.

\section{Random walks and inverted orbits}\label{s.inverted}

Let $T_i,i\in \mathbb Z$ be a positive integer-valued i.i.d.\ sequence. The two sided-random walk with increments $T_i$ has value $T_1+\ldots + T_i$ for $i\ge 1$ and $-T_{-1}-\ldots -T_i$ for $i<0$. Its image as $i$ ranges over $\mathbb Z$ is a random set called the two-sided renewal process. Its gap distribution is the law of $T_i$.

\begin{proposition}\label{p:rt_renewal}
Let $Y_i=G_1G_2\ldots G_i$ be a random walk on a group $G$, acting on a rooted set $(S,o)$. Let
\[
V_t=\{i\ge 0\; :\; o.Y_i^{-1}= o.Y_t^{-1}\},
\]
namely the times at which the inverted orbit visits the same point as at time $t$.  Then $V_t$ has the same distribution as $(t+\rho)\cap [0,\infty)$ where $\rho$ is the two-sided renewal process whose gap distribution is the law of the first return time of the Markov chain $o.Y_n$.
\end{proposition}

\begin{proof}
Fix n. Let $R\subset [1,t]$ and $S\subset [1, n-t]$ be sets of integer times, and consider the event $V_t\cap [0,n]=(t-R)\cup (t+S)$. We can write this event as the intersection of the events that
 \[
\{1 \le r\le t\;:\;o.Y_{t-r}^{-1}=o.Y_{t}^{-1}\}=R
\]
and
\[
\{1 \le s\le n-t\;:\;o.Y_{t+s}^{-1}=o.Y_{t}^{-1}\}=S.
\]
The equality in the first event can be written as
\[
o.G_{t-r}^{-1}\cdots G_1^{-1}=o.G_{t}^{-1}\cdots G_{1}^{-1}.
\]
Multiplying from the right by $G_1\cdots G_{t-r}$ we get
\[
o=o.G_{t}^{-1}\cdots G_{t-r}^{-1}.
\]
Similarly, the equality in the second event is equivalent to
\[
o=o.G_{t+1}\cdots G_{t+s}
\]
Clearly, these specify the return times of two independent Markov chains. Such return times are renewal processes, and the gap distribution is given by the time of the first return. To verify our claim, all we have to check is that the first return time $T$ for the walk with the inverse steps has the same distribution as that for the random walk with normal steps. Note that $T=k$ is equivalent to \begin{equation}\label{e:firstreturn}
o.G_1^{-1}\cdots G_{k}^{-1}=o,
\end{equation}
and $$o.G_1^{-1}\ldots G_\ell^{-1}\not= o, \mbox{ for all } \ell<k.$$ If the first equality holds, then the second one is equivalent to
$$o.G_{\ell+1}^{-1}\ldots G_k^{-1}\not= o, \mbox{ for all } \ell<k
$$or, multiplying over, to
$$o\not=o.G_{k}\ldots G_{\ell+1}, \mbox{ for all } \ell<k.$$
Now, together with \eqref{e:firstreturn}, this specifies the first return time for the walk with group increments in the reverse order: $G_k,\ldots , G_1$.
\end{proof}

Let $Q_n$ be the (random) occupation measure of the inverted orbit of the first $n$ steps.
More precisely, for $s\in S$ we have
$Q_n(s)=\sum_{i=0}^n {\one} (o.Y_i^{-1}=s)$, the number of
visits of the inverted orbit to $s$. We will often think of $n$ as fixed and call $Q_n$ simply the inverted orbit.
The support of $Q_n$ is will be called the range. Let $|Q_n|$ denote the
size of the range.

\begin{corollary}\label{c:Qn sum return} Let $T$ denote the first return time for the random walk
$o.Y_n$. Then we have
$$
\ev |Q_n| = \sum_{i=0}^n \pr (T>i) = \ev [T\wedge (n+1)].
$$
\end{corollary}
\begin{proof}
Let $B_i$ the event that time $i$ is the last visit to
$o.Y_i$ up to time $n$. Note that by Proposition \ref{p:rt_renewal}
$\pr(B_i)=\pr(T>{n-i})$. We have
$$\ev |Q_n|= \ev \sum_{i=0}^n \one_{B_i}=\sum_{i=0}^n \pr(B_i)
$$
Changing variables with $j=n-i$ we get
$$
\ev |Q_n|=\sum_{j=0}^n \pr(T>j)=\sum_{j=0}^\infty \pr(T\wedge (n+1)>j)
$$
so the last equality follows from the tail sum formula for
expectation.
\end{proof}
\begin{corollary} \label{c:Qlower} Let $T$ denote the first return time for the random walk
$o.Y_n$. Then for $0\le t \le n$ we have
$$
\pr( Q_n(o.Y^{-1}_t)\le k) \ge 1-2\pr(T\le n)^{k/2}.
$$
\end{corollary}
\begin{proof}
By Proposition \ref{p:rt_renewal}, we have
$$
\pr( Q_n(o.Y^{-1}_t)\le k) = \pr(A+B<k)
$$
where $A,B$ are the number of renewals up to and including time $t$
and $n-t$, respectively in two independent
renewal processes with gap distribution given by the law of $T$. This gives the lower bound
$$
\pr( Q_n(o.X_t)\le k) \ge 1-\pr(A\ge k/2) -\pr(B\ge k/2)
$$
If $A\ge k/2$, then the first $k/2$ gaps are all of size
at most $t\le n$. So, by the independence of gaps, we can bound
$$
\pr(A\ge k/2)\le \pr(T\le n)^{k/2}.
$$
The same argument for $B$ now concludes the proof.
\end{proof}

\section{Entropy and speed}\label{s.entropy and speed}
Recall that the {\bf entropy} of a finite non-negative
measure $\mu$ supported on a countable set $G$ is defined
by
\[
\Ent(\mu) = \sum_{x\in G} -\mu(x) \log\mu(x),
\]
where by convention $0\log 0=0$.
The entropy $H(X)$ of a discrete random variable $X$ is
given by the entropy of its distribution; the entropy
$H(X_1,\ldots X_n)$ of several random variables is given by
the entropy of the joint distribution of the $X_i$.
In
order to define conditional entropy for two random
variables $X,Y$, let $f(y)$ denote the entropy of
the conditional distribution of $X$ given $Y=y$. Then the
{\bf conditional entropy of $X$ given $Y$} is defined as
$H(X|Y) := \E f(Y)$.

The conditional entropy satisfies
\[
H(X,Y) = H(X|Y) + H(Y).
\]
A useful and easy fact is that among measures supported on a given
finite set, the one having maximal entropy is the uniform measure
on that set.

Let $h(n)$ be the entropy of the $n$-step random walk in the group $\Lambda$. Recall that $h$ is concave (\cite{KV}, Proposition 1.3) and hence by linear approximation we can extend it to a concave function on $\mathbb R_+$.

\begin{lemma} \label{l:concave} For every $n\ge 1$ the function $x\mapsto xh(n/x)$ is concave.
\end{lemma}
\begin{proof}
By the concavity of entropy, we have
$$
\frac{x}{x+y} h(n/x) +  \frac{y}{x+y} h(n/y) \le h\left( \frac{x}{x+y} n/x +  \frac{y}{x+y} n/y\right) = h\left(\frac{2n}{x+y}\right)
$$
equivalently,
$$
\frac{xh(n/x) +  yh(n/y)}{2} \le \frac{x+y}{2}h\left(\frac{2n}{x+y}\right)
$$
which is the claimed concavity.
\end{proof}

The following is a precise version of the well-known relationship between entropy and expected distance (\cite{ErschlerZ2}) that follows from the bounds of \cite{varoplongrange} and \cite{carne}. The proof we present below is due to Yuval Peres. 
\begin{proposition}\label{p:VC}
Consider a random walk $X_n$ on a graph started at a vertex $o$, and let $|\cdot|$
denote graph distance from $o$.  Then we have
$$
\ev |X_n| \leq \sqrt{2n\left(H(X_n) +  \log 2\sqrt{\mbox{$\max_v \frac{\deg v}{\deg o}$}}\right)}.
$$
\end{proposition}
\begin{proof}
\begin{equation}\label{e:vc1}
H(X_n) =  -\sum_{v\in V}\pr (X_n=v)\log \pr (X_n=v).
\end{equation}
By the bounds of \cite{varoplongrange} and \cite{carne}, with $\eta=\max_v \frac{\deg v}{\deg o}$ and $\kappa=\log (2\sqrt\eta)$ we have
$$
\log \pr(X_n=v)\le \kappa-|v|^2/2n.
$$
Substituting this into \eqref{e:vc1} and using the fact that $\sum_{v\in V}\pr(X_n=v)=1$ we get
$$
H(X_n) \ge \sum_{v\in V}\pr(X_n=v)\frac{|v|^2}{2n} -\kappa = \frac{\E(|X_n|^2)}{2n} - \kappa \geq \frac{\E^2(|X_n|)}{2n} -\kappa.
$$

The claim follows by rearranging the terms and taking square roots.
\end{proof}

\section{Speed in permutation wreath products}\label{s.speed in pw}
We consider the following setup. Let $\Gamma$ be a finitely generated countable infinite group acting on a set $S$. We single out an element $o\in S$ and call it the root.
Let $\mu$ be a finitely supported symmetric measure on $\Gamma$. Let $\Lambda$ be a finitely generated countable (possibly finite) group. The \textbf{permutation wreath product} $\Lambda \wr_S \Gamma$ is the semidirect product of $\Lambda^S$ with $\Gamma$ acting on it by permuting the coordinates. The multiplication rule, for $\ell,\ell'\in (\Lambda^S)$ and $g,g'\in \Gamma$ is
$$
(\ell,g)(\ell',g')=(\ell\ell'^{g^{-1}},gg')
$$
where $\ell'^{g^{-1}}$ is defined by $\ell'^{g^{-1}}(s)= \ell'(s.g)$.

A \textbf{switch} is a random element of $\Lambda \wr_S \Gamma$ of the form $(\bar L,id_\Gamma)$, where
$$
\bar L(s)=\begin{cases}id_\Lambda \text{ if } s\not=o, \\
L \text{ if } s=o,
\end{cases}
$$
where $L$ is a random elemet of $\Lambda$ chosen from a fixed symmetric finitely-supported measure.

We consider the random walk
$$
X_n= \prod_{i=1}^n \bar L_i G_i \bar L'_i
$$
(called the \textbf{switch-walk-switch} random walk) on the permutation wreath product. Here the $G_i$ are independent choices from the measure on $\Gamma$, and the $L_i,L_i'
$ are independent choices from the measure on $\Lambda$. We have $X_i=(\cdot, Y_i)$ where $Y_i=G_1\cdots G_i$.

Keeping with the tradition for ordinary wreath products, we call the group $\Lambda \wr_S \Gamma$ the Lamplighter group, and the above walk the lamplighter walk.

Unlike the usual lamplighter group, where there is a walker on the (Cayley) graph who keeps switching lights along its path, this walk cannot be interpreted as a lamplighter moving along the Schreier graph of the action $G$ on $S$. This is because the switches happen at locations $o$, $o.G_1^{-1}$, $o.G_2^{-1}G_1^{-1}$, $o.G_3^{-1}G_2^{-1}G_1^{-1}$, a sequence that is not necessarily Markovian (or even connected in the Schreier graph). Note that $o$, $o.G_1^{-1}$, $o.G_1^{-1}G_2^{-1}$, $o.G_1^{-1}G_2^{-1}G_3^{-1}$ would be a Markovian random walk on the Schreier graph, but our case is different. Following \cite{BartholdiErschler} we call the sequence $o$, $o.G_1^{-1}$, $o.G_2^{-1}G_1^{-1}$, $o.G_3^{-1}G_2^{-1}G_1^{-1},\ldots$ the inverted orbit of $o$ under the walk $Y_n$ (or alternatively of the walk $X_n$ as they have the same action on $S$.) In particular, we get the useful observation, which will be exploited later, that the inverted orbit of $X_n$ depends only on $Y_n$ and not on the lamp group or the switch steps chosen.

A symmetric finitely-supported measure on a group defines a Cayley graph. Let $|g|$ denote the corresponding graph distance of $g$ from the identity.

With $L_i$ random walk increments on $\Lambda$ as above, define  $$\underline \lambda(t)=
\inf_{n\ge t}\ev |L_1\cdots L_n|, \qquad \overline\lambda(t)=
\max_{n\le t}\ev |L_1\cdots L_n|.$$
A simple computation checks that both $\underline \lambda$ and $\overline \lambda$ are subadditive: $\lambda(t+s)\le \lambda(t)+\lambda(s)$.
Also, unlike the ordinary
expected distance, both $\underline \lambda$ and $\overline \lambda$ are nondecreasing (by definition). Lemma 4.1 in \cite{LeePeres} implies that \begin{equation}\label{e:monspeed}
\frac{1}{2}(\bar \lambda(n)-1) \le \ev |L_1\cdots L_n| \le 2 \underline \lambda(n)-1
\end{equation}

The following theorem connects speed on $\Lambda$ to that on the permutation wreath product. The main feature is that it only uses the tail of the first return time $T$ of the random walk on the Schreier graph of $(S, \Gamma)$. We emphasize that $T$ relates to the ordinary random walk.

\begin{theorem}\label{t:speed_lb} The random walk $X_n$ on the permutation wreath product satisfies for $n\ge 1$
$$
\ev |X_n| \ge
\frac{np\underline \lambda(1/p)}{16}.
$$
for every $p\le \pr(T>n)$.
\end{theorem}

\begin{proof}
Note that the number of switches of the light at $s$ by time $n$
given by
\begin{equation}\label{e:steps}
\widetilde Q_n(s)= 2Q_n(s)-\one_{\{o\}}(s) -\one_{\{o.Y_n^{-1}\}}(s)\ge Q_n(s)
\end{equation}
 for $n\ge 1$. Sice
all these lights have to be turned off to return to the identity, we
have
$$
\ev |X_n| \ge \ev \sum_{s\in S} \underline \lambda(Q_n(s))
$$
Let $k$ be a constant whose value will be decided later on. Call an element $s\in S$ thin if  $Q_n(s)\le k$, and call a time
$0\le t\le n$ thin if $o.Y_t^{-1}$ is thin. We can bound the right hand side below as
$$
\ev \sum_{s\in S, s \text{ thin}} \underline \lambda(Q_n(s)).
$$
Let $N$ be the number of thin times. We want to bound the quantity in the expectation in terms of  $N$.
This gives the following minimization problem. Minimize $\sum \underline \lambda(x_i)$ given that
$\sum x_i=N$ and the $x_i<k$ are nonnegative integers.
Note that since $\underline \lambda$ is subadditive,
we can decrease the quantity in question without violating the constraints if we replace
two $x_i$'s that are less than $k/2$ by their sum and zero.
After repeating this procedure, we get that all but perhaps one of the nonzero $x_i$ are greater or equal than $k/2$. Let $j$ be the number of
$x_i$ at least $k/2$. Then we have $jk+k/2\ge N$, so $j\ge N/k-1/2\ge N/(2k)$.
This gives
$$
\sum \underline \lambda(x_i)\ge \frac{N}{2k}\underline \lambda(k/2)
$$
substituting this into the above, we get
\begin{equation}\label{e:evN}
\ev |X_n| \ge  \ev \frac{N}{2k}\underline \lambda(k/2)=\frac{\ev N}{2k}\underline \lambda(k/2)
\end{equation}
it remains to compute
$$
\ev N=\sum_{t=0}^n \pr(Q_n(o.Y_t^{-1})\le  k)\ge n(1-2(1-\pr(T>n))^{k/2}))\ge n(1-2(1-p)^{k/2}))
$$
by Corollary \ref{c:Qlower}. Setting $k=2/p$ gives the lower bound of $n/4$, and substituting into \eqref{e:evN} we get
$$
\ev |X_n| \ge  \frac{np\underline \lambda(1/p)}{16}
$$
as claimed.
\end{proof}

Analogous to Theorem \ref{t:speed_lb} about speed, we have an upper bound for entropy in the wreath product. Again, the relation involves the return time of the random walk in the Schreier graph, since by Corollary \ref{c:Qn sum return}
$$
\ev |Q_n| = \sum_{i=0}^n \pr (T>i) = \ev [T\wedge (n+1)].
$$

\begin{theorem}\label{t:Ent_upper_general} We have
\begin{equation}\label{e:ent_upper_general}H(X_n) \le H(Y_n) + H(\supp Q_n|\,|Q_n|) + 2 \ev |Q_n|\left(h\left(\frac{n}{\ev |Q_n|}\right) + \log (n+1) \right) .
\end{equation}
And if $\Lambda$ is finite than also
\begin{equation}\label{e:finitelamp} H(X_n) \le H(Y_n) + \ev |Q_n|\log|\Lambda| + H(\supp Q_n).
\end{equation}
and in both clauses the expression on the right is increasing as a function of $\ev |Q_n|$.
\end{theorem}

\begin{proof}
We start with the general case.
Write $X_n=(\mathcal L_n, Y_n)$, a pair containing a lamp configuration and a walk $Y_n$ on $\Gamma$.
\begin{equation}\label{e:ent1}
H(X_n) \le H(Y_n)+H(\mathcal L_n) \le H(Y_n)  +H(\mathcal L_n|Q_n)+ H(Q_n)
\end{equation}
We first consider the entropy of $Q_n$.
\begin{eqnarray}
H(Q_n) &=& H(\supp Q_n ) +H(Q_n|\supp Q_n) \notag \\
&\le& H(\supp Q_n|\, |Q_n|) + H(|Q_n|) + \ev |Q_n|\log (n+1) \label{e:ent2}
\end{eqnarray}
the second inequality follows since $1\leq Q_n(s) \leq n+1$ for every $s\in \supp Q_n$. Further, the random variable $|Q_n|$ is between $1$ and $n+1$, so its entropy is at most $\log(n+1)$.

We now turn to $H(\mathcal L_n|Q_n)$. Note that given $Q_n$, each entry of $\mathcal L_n(s)$ is distributed as independent samples from a random walk on $\Lambda$ of length $\widetilde Q_n(s) \le 2Q_n(s)$ (see
\eqref{e:steps}), which is at most
$$
H(\mathcal L_n|Q_n) \le \ev \sum_s h(\widetilde Q_n(s))
$$
now we use the monotonicity, concavity, and subadditivity of $h$ (\cite{KV} propositions 1.1 and 1.3)  to get
$$
\sum_s h(\widetilde Q_n(s)) \le \sum_s 2h(Q_n(s)) \le 2|Q_n| h(n/|Q_n|)
$$
By Lemma \ref{l:concave}, we have that $x\mapsto xh(n/x)$ is concave, so by Jensen's inequality
$$
H(\mathcal L_n|Q_n)\le  2\ev |Q_n| h(n/\ev |Q_n|).
$$
We summarize
$$H(X_n) \le H(Y_n) + 2 \ev |Q_n|h(n/\ev |Q_n|) + H(\supp Q_n|\, |Q_n|) + \ev |Q_n| \log (n+1) + \log (n+1).
$$
Using $|Q_n|\ge 1$ the claimed inequality follows. Since the function $x\mapsto xh(n/x)$ is concave and nonnegative, it must be increasing.

If the lamp group $\Lambda$ is finite, one may replace \eqref{e:ent1} by $$ H(X_n) \le H(Y_n)+H(\mathcal L_n) \le H(Y_n)  +H(\mathcal L_n|\supp Q_n)+ H(\supp Q_n).$$ Equation \ref{e:finitelamp} follows since every word in $\Lambda$ can be described by $\log |\Lambda|$ bits.
\end{proof}

\begin{remark} \label{r:finitelampgroup}\
 Given the size $|Q_n|$ of the support of $Q_n$, the support can be described by a sequence of $|Q_n|$ elements from the ball $B_n(o)$ of radius $n$ in the Schreier graph of $(S,\Gamma)$. This gives $$H(\supp Q_n|\, |Q_n|)\le \ev |Q_n|\log |B_n(o)|.$$
  However, as we will see in Section \ref{s:assembly}, for the groups we will be interested in one can give a better bound, linear in $|Q_n|$.
\end{remark}

The following lemma estimates the size of the inverted orbit of a random walk on a Schreier graph in terms of the speed of random walk on the permutation wreath product with $\mathbb Z_2$. Together with Theorem \ref{t:upper_tight} it usually gives a better upper bound on the speed (of the random walk $X_n$ on $\Lambda \wr_S G$) than applying the Varopolous-Carne bound of Proposition \ref{p:VC} directly to the entropy bound of Theorem \ref{t:Ent_upper_general}.

\begin{lemma}\label{l:Gady}  Let $Y_n$ be a random walk on a group $G$ acting on a rooted set $(S,o)$. Let $Q_n$ denote the inverted orbit of $Y_n$. Consider the switch-walk-switch random walk $X_n'=({\mathcal L}_n',Y_n)$ on the permutation wreath product $\mathbb Z_2\wr_S G$. There exists a word in $G$ that is equivalent to $Y_n$ whose inverted orbit includes\, $\supp Q_n$, and whose expected length is at most $3\ev |X_n'|$.
\end{lemma}

\begin{proof}
Note that given $Q_n$, each lamp in $\supp Q_n$ is turned on or off with probability half each, independently.

If $\mathcal L_n'$ is the lamp configuration at time $n$, then given $Q_n$, the distribution of $\mathcal L_n'$ and the inverted configuration $\mathcal L_n''$ defined as
$$\mathcal L_n''(s)=\one(Q_n(s)>0)-\mathcal L_n'(s)$$
is the same. Let $X_n''=(\mathcal L_n'',Y_n)$. By the above, $X_n''$ has the same distribution as $X_n'$.

Since in $X_n'$ all the lamps in $\supp \mathcal L_n'$ are turned on, there is a word $W'$ of length at most $|X_n'|$, equivalent to $X_n'$, whose inverted orbit includes  $\supp\mathcal L_n'$. Similarly, there is a word $W''$ of length at most $|X_n''|$ whose inverted orbit includes $\supp \mathcal L_n''$. Then the inverted orbit of the word $W''{W''}^{-1}W'$ contains $\supp Q_n$,  it is equivalent to $X_n'$ and its expected word length is at most
\[2\ev |X_n''|+\ev|X_n'| =3\ev |X_n'|. \qedhere\]
\end{proof}

\begin{theorem}\label{t:upper_tight}
For every $q\ge \ev |Q_n|$ we have
$$
\ev |X_n|\le 3\overline \lambda\left(n/q\right)q
+12\sqrt{n} \sqrt{H(Y_n) + H(\supp Q_n) + \ev |Q_n| }.
$$
\end{theorem}
\begin{proof}
We have $|X_n|\le A+B$, where $A$ is the length of the shortest word
equivalent to $Y_n$ whose inverted orbit contains $\supp Q_n$ and $B$ is the number of lamp steps required to turn off the lamps.

We first bound $B$. Given $Q_n$, the expected number of lamp moves required to turn off the lamps is bounded above by
$$
\ev[B|Q_n]\le \sum_{v\in \supp Q_n} \overline \lambda(Q_n(v))
$$
first we use the monotonicity of $\overline \lambda$ to round up each argument to the next integer multiple of $k=n/q$. The total sum of the new arguments is at most the total sum of the originals ones (i.e. 2n) plus $|Q_n|k$. Then we use subadditivity to bound the sum by
$$
\overline\lambda(k)\frac{2n+|Q_n|k}k = \overline\lambda(k)(2q+|Q_n|)
$$
Taking expectations we get
$$
\ev B\le \overline \lambda\left(n/q\right)3q.
$$
Regarding $A$, Lemma \ref{l:Gady} gives us
$$
\ev A \le 3 \ev |X'_n|,
$$
where $X'_n$ is the switch-walk-switch random walk on $ \mathbb Z_2\wr_S G$. For this, we use the entropy bound \eqref{e:finitelamp} and the Varopoulos-Carne bounds in Proposition \ref{p:VC} to get
$$\ev |X_n'| \le  4\sqrt{n} \big(H(Y_n) + H(\supp Q_n) + (\log 2)\ev |Q_n|\big)^{1/2}
$$
This completes the proof.
\end{proof}

We summarize the results of this section and deduce some conclusions using the Varopoulos-Carne bounds in Proposition \ref{p:VC}. The second upper bound on the speed is usually sharper.
\begin{corollary}\label{c:spedd_in_pwp}
For every $p\le \pr(T>n)$ and for every $q\ge \sum_{i=0}^n \pr (T>i)$ we have
\begin{eqnarray}
\ev |X_n| &\ge& \frac{np\underline \lambda(1/p)}{16}, \label{e:speedlower} \\
\ev |X_n| &\le&  4\sqrt{n} \big(H(Y_n) + H(\supp Q_n|\, |Q_n|) + 2
q\left(h\left(n/q\right) + \log (n+1)\right)\big)^{1/2}\,\notag
\\
\ev |X_n|&\le &3\overline \lambda\left(n/q\right)q +12\sqrt{n} \sqrt{H(Y_n) + H(\supp Q_n) + \ev |Q_n|},\label{e:speedupper}
\\
H(X_n) &\ge& \frac{np^2\underline \lambda^2(1/p)}{4096}, \notag
\\
H(X_n) &\le&  H(Y_n) + H(\supp Q_n|\, |Q_n|) + 2
q\left(h\left(n/q\right) + \log (n+1)\right). \notag
\end{eqnarray}
\end{corollary}

\begin{proof}[Proof of Theorem \ref{t:maingeneral}]
The lower bound follows from \eqref{e:speedlower} and the almost monotonicity of speed \eqref{e:monspeed}. The bound \eqref{e:speedupper}, our assumptions \eqref{e:cond1}, and \eqref{e:monspeed} yields
$$
\ev |X_n| \le cn(p_n\lambda(1/p_n) + \sqrt{p_n}).
$$
Now for a generating random walk on any transitive graph $\lambda(n) \ge c \sqrt{n}$, see \cite{LeePeres}. The upper bound follows.
\end{proof}

\section{The assembly line}\label{s:assembly}

The goal of this section is to describe an analyze a simple birth-and-death Markov chain, called the {\bf assembly line},
which is the core of our construction. In the next section, we will show that it essentially a Schreier graph of a group. However, the first part of the analysis does not need this fact.

Fix a bounded sequence of integers $m_i\ge 2$, and consider the set $S$
of left-infinite words $\dots w_3w_2w_1$, where
$w_i\in \{0,\ldots,m_{i-1}\}$, and $w_i$ are 0 for all $i$ large enough.
We consider the following random walk
on this set. In each step, toss a fair coin; if heads, randomize $w_1$; if tails, randomize the letter $w_{i+1}$ after the first nonzero letter $w_{i}$.

This can be thought of as an infinite assembly line, where $w_i$
is the state of the worker at position $i$, either dozing
off ($w_i=0$) or doing one of the $m_i-1$ possible jobs. The
foreman, standing in front of the line, has the option
to nudge the first worker, or shout at the first awake
worker in the line to nudge the worker after them. Nudged
workers just change their state to one chosen uniformly at random.

This is a reversible random walk on $S$, and its graph
(as we will see) is the Schreier graph of the action of a group denoted $\mm$. In order to analyze the speed in the permutation wreath product $\Lambda \wr_S \mm$, we just have to understand the
first return time $T$ of this walk started at $o=\ldots 000$.

Fortunately, this walk projects to a birth-and-death Markov chain. Indeed, consider the projection from $S$ to the set of finite binary strings
$$
\ldots w_3w_2w_1\mapsto  \ldots\bar w_3\bar w_2\bar w_1, \qquad \bar w_i=\one(w_i>0).
$$
The projection of the walk has the same description as above, except that ``randomize'' for the $\bar w_i$ means
$$
\bar w_i :=\begin{cases} 0&\mbox{with probability } 1/m_i\\
1&\mbox{with probability } 1-1/m_i.
\end{cases}
$$ Note also that from each $\bar w\in \bar S$, there are only two possible new positions one can move to (except $o$, where there is only one). This implies that the graph structure of this walk is that of the nearest-neighbor graph of natural numbers.
Indeed, the structure is a standard Gray code -- an ordering of all binary strings so that consecutive strings differ by only one bit.

It is easy to that the position of a given string in this order is
$$\mbox{position}(\ldots b_3b_2b_1)=\ldots \overleftarrow{b_3}\overleftarrow{b_2}\overleftarrow{b_1},\qquad
\mbox{where }\overleftarrow{b_i}=b_i+b_{i+1}+\ldots  \mod 2,
$$
where the last sum has only finitely many nonzero terms.

Thinking of the walk on $S$ as a random walk on a weighted
graph, the uniform stationary measure on $S$ projects
to a stationary measure on $\bar S$. More precisely,
$$
\pi(\ldots b_3b_2b_1)=\prod_{i=1}^\infty (m_i-1)^{b_i}
$$
and the edge weights are within constants of the vertex weights. This automatically gives that the resistance between $o$ and $2^n$ in $\bar S$ satisfies
\begin{equation}\label{e:res}
r_\ell\le \res(o,2^\ell)\le 2m_* r_\ell,
\end{equation}
where $r_\ell$ is the sum of the inverse vertex weights for vertices between these two positions, and $m_*$ is the maximum degree.  That is,
$$
r_\ell = \sum_{b_1,\ldots, b_n} \pi^{-1}(b_n\ldots b_1)=\sum_{b_1,\ldots, b_n}  \prod_{i=1}^\ell \left(\frac1{m_i-1}\right)^{b_i} =\prod_{i=1}^\ell \left(\frac {1}{m_i-1}+1\right)
$$
since the sum factorizes. In short,
$$
\res(o,2^\ell) \asymp r_\ell:= \prod_{i=1}^\ell \frac {m_i}{m_i-1}
$$
Define $v_\ell = m_1\cdots m_\ell$. Note that the ball of radius $2^\ell$ in $S$ about $o$ contains exactly $v_\ell$  vertices.

\begin{proposition}\label{p:return_sum_ub}
For $\ell>0$, we have $$\sum_{i=0}^{v_\ell r_\ell} \pr(T>i) \le 2 v_\ell.$$
\end{proposition}
\begin{proof}
We truncate the graph at $2^\ell$, and consider the random walk on the truncated graph together with the random walk on the original graph. These are naturally coupled until the hitting time $\tau$ of $2^\ell$. Let $T'$ denote the first return time in the truncated graph. Then we have
$$
T>i \;\; \Rightarrow \;\; T\wedge \tau>i \mbox{ or } T>\tau \;\; \Rightarrow \;\; T'>i \mbox{ or } T>\tau.
$$
So we have (by bounding the $i=0$ case by 1)
$$\sum_{i=0}^{n} \pr(T>i) \le 1+
\sum_{i=1}^\infty \pr(T'>i) + n\pr(T>\tau).
$$
The sum of the first quantities equals $\ev T'=v_\ell$ (since in any connected weighted graph the expected return time to a vertex equals the total volume over the weight of the vertex). Also $\pr(T>\tau) = 1/\res(0,2^n)\le 1/r_\ell$. The claim follows if we set $n=v_\ell r_\ell$.
\end{proof}

\begin{lemma}\label{l:RW_notsmall}
Let $X$ be nonnegative interger-valued random variable with expectation $a$ satisfying $\ev[X-x|X\ge x]\le a$ for all integers $x\ge 0$. Then
$$
\pr(X> a/4)\ge 1/31.
$$
\end{lemma}
\begin{proof}
By Markov's inequality for all integers $t$ we have
$$
P[X\ge t+ea|T>t]\le 1/e
$$
This implies that for all integers $k$
$$
P[X\ge kea]\le e^{-k}
$$
and so for all positive real numbers $k$
$$
P[X\ge kea]\le e^{-k+1}
$$
This gives
$$
\ev X - \ev[X\wedge ra] =\int_{ra}^\infty P[X\ge x]\,dx \le e\int_{ra}^\infty e^{-x/(ae)}\,dx = ae^2e^{-r/e}
$$
With $r=8$ we get
$$
\ev X - \ev[X\wedge 8a] \le \ev X/2$$
and therefore
$$
\ev[X\wedge 8a]\ge \ev X/2
$$
Applying Markov's inequality to $8a-X\wedge 8a$
we get
$$
\pr(8a-X\wedge 8a \ge ka) \le \frac{7.5a}{ka}=7.5/k
$$
setting $k=7.75$ we get
$$
\pr(X> a/4)\ge 1-\frac{7.5}{7.75}=1/31.
$$
\end{proof}

\begin{lemma}\label{l:Expexted Hitting lb}
Consider the hitting time $T'$ of $o$ of the random walk
on the graph truncated at $2^\ell$. We have
$$
\ev_{2^\ell} T' \ge r_{\ell-1}v_{\ell-1}(m_\ell-1).
$$\end{lemma}
\begin{proof}
This is a birth-and-death chain, one of the simplest kind of Markov chains, and the expected hitting times can be computed explicitly.
By \cite{LPw} (page 27, formula (2.16)), we have
$$
\ev_{2^\ell} T' =\sum_{i=1}^{2^\ell} \frac1{p(i-1,i)\pi_i}\sum_{j=i+1}^{2^\ell}\pi_j\ge
\sum_{i=1}^{2^\ell-1} \frac1{\pi_i}\sum_{j=2^{\ell-1}+1}^{2^\ell}\pi_j=r_{\ell-1}v_{\ell-1}(m_\ell-1).$$
\end{proof}

\begin{proposition}\label{p:return_lb}
$$\pr(T>r_{\ell-1} v_{\ell-1}/4) \ge \frac1{62 m_* r_\ell}.$$
\end{proposition}
\begin{proof}

We have
\begin{equation}\label{e:decomp}
\mbox {walk hits }2^\ell\mbox{ before }o\mbox{ and }T^*>t \Rightarrow T>t,
\end{equation}
where $T^*$ is the time to hit $o$ after the first time of hitting $2^\ell$. The probability of the first event
is given by $1/\res(0,2^l)\ge 1/(r_\ell 2m_*)$ by \eqref{e:res}. We now consider $T^*>t$. Note that
$T^*$ stochastically dominates the hitting time $T'$ from $2^l$ to $o$ in the graph truncated at $2^l$. This is because
we $T^*=T'+T''$ where $T''$ is the number of steps outside
$[0,2^\ell]$ before hitting $0$. Note that $T'$ also stochastically dominates the hitting times of $o$ in the truncated graph when starting from lower vertices. Thus Lemma \ref{l:RW_notsmall} gives
$$
\pr(T'> \ev T'/4)\ge 1/31.
$$
Using the lower bound on $\ev T'$ from Lemma \ref{l:Expexted Hitting lb} we get for $t= r_{\ell-1}v_{\ell-1}/4$ we have
$
\pr(T'> t)\ge 1/31.
$
By the strong Markov property and \eqref{e:decomp} we get
$$
\pr(T>t)\ge \frac{1}{62 m_* r_\ell}.
$$
\end{proof}
Recall that
$$
r_\ell= \prod_{i=1}^\ell \frac {m_i}{m_i-1}, \qquad v_\ell = m_1\cdots m_\ell,\qquad  n_\ell= r_\ell v_\ell
$$

We define the inverse function and an exponent
$$
\ell(n) =  \min \{\ell \;:\: n_\ell \ge n\}, \qquad \alpha_n =
\frac{\log v_{\ell(n)}}{\log n}
$$

Note that
\newcommand{\alphan}{{\alpha_n}}
\begin{eqnarray*}
v_{\ell(n)} &=&n^\alphan,
\\
 n
   \;\;\le\;\; n_{\ell(n)}&\le& 2m_*n,\\
 n^{1-\alphan}
  \;\;\le\;\; r_{\ell(n)}
  &\le& 2m_*n^{1-\alphan}.
\end{eqnarray*}
We express the results of Propositions \ref{p:return_sum_ub} and \ref{p:return_lb} in terms of $n$.
\begin{corollary}\label{c:return_probs_wreath}
\begin{eqnarray}\label{e:low}
\pr(T>n)&\ge&  \frac{1}{500 m_*^2}\,  n^{\alphan-1},
\\ \sum_{i=0}^{n}  \pr(T>i)&\le &
2n^\alphan.\label{e:up}
\end{eqnarray}
The combination implies a matching upper bound on \eqref{e:low} and lower bound on \eqref{e:up} up to a constant factor depending on $m_*$ only.
\end{corollary}
\begin{proof}
Since $r_{\ell'-1} v_{\ell'-1}/4\ge r_{\ell'-2} v_{\ell'-2}$, we set $\ell'=\ell(n)+2$ to get
$$
\pr(T>n)\ge \pr(T>n_{\ell(n)}) \ge\pr(T>r_{\ell'-1}v_{\ell'-1}/4)  \ge \frac1{62 m_* r_{\ell'}} \ge  \frac{1}{500 m_*^2}\,  n^{\alphan-1}.$$
since
$$
r_{\ell'}=r_{\ell(n)+2}\le 4r_{\ell(n)} \le 8m_* n^{1-\alphan}.
$$
For the upper bound, we have
 \[\sum_{i=0}^{n}  \pr(T>i)\le \sum_{i=1}^{n_{\ell(n)}} \pr(T>i) \le 2 v_\ell(n)=2n^\alphan.\qedhere \]
\end{proof}

We finish this subsection by showing that one can get good control over the sequence $\{\alphan\}$
by choosing the appropriate degree sequence $\{m_i\}$.

\begin{lemma}\label{l:choosing m}
For any $\gamma\in[1/2,1)$ and any function $f:\R_+\rightarrow \R_+$ satisfying
$f(1)=1$ and the log-Lipshitz condition that for all real $a,n\ge 1$
$$a^{1/2}f(n)\leq f(an)\leq a^\gamma f(n)$$
 there is
a bounded sequence $(m_\ell)$ such that for all $n\geq 1$ we have
$$
{c_\gamma^{-1}}\le \frac{f(n)}{n^{\alpha_n}} \leq 2c_\gamma, \qquad \mbox{ with } c_\gamma=3e^{1/(1-\gamma)}.
$$
\end{lemma}

\begin{proof}
Note that
$$
\frac{f(n_{\ell+1})}{v_{\ell+1}}=\frac{f(n_{\ell})}{v_{\ell}}
y\qquad \mbox{ where } y=\frac{f(n_{\ell}m_{\ell+1}^2/(m_{\ell+1}-1))}{f(n_\ell)}\;\frac1{m_{\ell+1}}.
$$
By the assumptions on $f$ we have
\begin{equation}\label{e:fn}
\frac{1}{\sqrt{m_{\ell+1}-1}} \le y \le \left(\frac{m_{\ell+1}^2}{m_{\ell+1}-1}\right)^\gamma \frac{1}{m_{\ell+1}}
\end{equation}
Set $m_*$ be the smallest integer $m_{\ell+1}$  so that the upper bound at is not more than 1. This depends on $\gamma$ only. Now, for $\ell\ge 1$ inductively choose
$$
m_{\ell+1}=
\begin{cases}
2 & \mbox{if } f(n_\ell)/v_\ell\le 1\\
m_* & \mbox{if } f(n_\ell)/v_\ell> 1
\end{cases}
$$
then \eqref{e:fn} and induction give the inequality
$$
\frac{1}{\sqrt{m_*-1}} \le \frac{f(n_\ell)}{v_\ell} \le 2^{2\gamma-1}, $$
which holds trivially for $\ell=0$.

Equivalently, by definition of $\alpha_n$,  the inequality
$$
\frac{1}{\sqrt{m_*-1}} \le \frac{f(n)}{n^{\alpha_n}}\le 2^{2\gamma-1}<2,
$$
holds at the values $n=n_\ell$.

To get the claim for all integers, recall that $n^{\alpha_n} =
v_{l(n)}$ is constant between $n_\ell$ and $n_{\ell+1}$ while
$f$ is increasing and may change by a factor at most $m_*$ between $n_\ell$
and $n_{\ell+1}$, so for all $n \ge 1$ we get
\[
\frac{1}{\sqrt{m_*-1}} \le \frac{f(n)}{n^{\alpha_n}}<2m_*.
\]
For the value of $c_\gamma$ note that the right hand side of \eqref{e:fn} is at most one when $m_{\ell+1}\ge e\times e^{1/(1-\gamma)}$, and $\lceil e\times e^{1/(1-\gamma)}\rceil \le 3e^{1/(1-\gamma)}$.
\end{proof}

\section{The piecewise mother groups}\label{s:p mother groups}

The Markov chain random walk on the ``assembly line'' introduced in Section \ref{s:assembly} is a random walk on a Schreier graph of a certain group that we will now introduce.

Given a sequence $m=\{m_i\}_{i\geq 1}$  of integers $m_i\geq 2$, we define $\tm$, the spherically symmetric rooted tree with degree sequence $\overline{m}$ to be the following graph. The
vertex set consists of all finite sequences $w_k\cdots w_1$ where $w_i\in\{0,\ldots,m_i-1\}$, where the empty sequence $o$ is the root. The edge set consists of all pairs of vertices of the form $\{w_k\cdots w_1,\,w_{k-1}\cdots w_1\}$.

Let $\Aut(\tm)$ be the set of rooted automorphisms of $\tm$, i.e.  automorphisms fixing the root.
The groups $\mm$ will be subgroups of $\Aut(\tm)$.

Before defining the groups $\mm$ let us say a few words on automorphisms of $\tm$.

Automorphisms $\gamma\in \Aut(\tm)$ can be written as a product
$$
\gamma=\langle \gamma_0,\ldots, \gamma_{m_1-1}\rangle \sigma
$$
where $\sigma\in \text{Sym}(m_1)$ permutes the subtrees of $o$ and $\gamma_i$
are automorphism of the subtrees. Thus the $\gamma_i$ are elements of $\Aut(\ttm)$ where $\ttm$ is the tree with the shifted degree sequence $\tau m:=m_2,m_3,\ldots$. The natural action of the group $\Aut(T_m)$ on infinite strings (the boundary of the tree) $\cdots w_3w_2w_1$ can be defined recursively by $$\cdots w_3w_2w_1.\gamma=(\cdots w_3w_2.\gamma_{w_1}) w_1.\sigma$$

More generally, any automorphism can be written as a product of automorphisms $\gamma_v$ of subtrees of vertices $v$ at level $\ell$ and a permutation permuting vertices at level $\ell$. The $\gamma_v$ are uniquely determined and are called the {\bf sections} of $\gamma$ at $v$.

We will be interested in two special types of automorphisms:
\begin{enumerate}
\item {\bf Permutations:}  automorphisms of the form $\mathbf{\sigma} =\langle id,id,\ldots,id \rangle\sigma$ which simply permute the children of the root according to $\sigma$, with all other sections trivial.
\item {\bf Propagating actions:} For every $\ell \ge 1$ fix some permutation $\sigma_\ell\in \text{Sym}(m_{\ell})$. Then we define
    recursively
    $$
    a_\ell=\langle a_{\ell+1},\sigma_{\ell+1},\cdots,\sigma_{\ell+1}\rangle\,id
    $$
\end{enumerate}
The name "propagating action'' comes from viewing the sections of these automorphisms as "propagating'' from the root along the zero ray, sending permutations to the neighbouring vertices of the tree. To see that the recursive definition does indeed define an automorphism of $\tm$ simply note that the image of each vertex in $\tm$ can be calculated in finitely many steps.

Consider a finite group $H$ of propagating actions with the property that $H^{-1}=H$ and for uniform random $h\in H$ the number $0.\sigma_\ell^{(h)}$ is uniformly distributed on $\{0,\ldots, m_{\ell-1}\}$. This is a variant of automorphisms considered by \cite{BrieusselEntropy}).

A simple example is when $H$ is the set of all propagating actions where $\sigma_\ell$ depends only on $m_\ell$ (e.g. all cyclic shifts of $\{0,\ldots,m_\ell-1\}$); in the rest of the paper we will assume that this is the case; this choice affects only the constants in our bounds.

We define $\mm$ as the group generated by the propagating actions $H$ and the permutations $\text{Sym}(m)$.

We consider the random walk $Y_n$ on $\mm$ where the step distribution is the even mixture of uniform measures on $\text{Sym}(m_1)$ and $H$. By definition of $H$, the random walk $o.Y_n$ on infinite strings (with finitely many nonzero letters) corresponds to the assembly line walk introduced in Section \ref{s:assembly}.

In order to prove an upper bound on the entropy of the walk $X_n$ (which will then translate to a bound on the speed of $X_n$),
it will be useful to introduce a tree structure to the inverted orbit.

Every automorphism $\gamma\in \Aut(\tm)$ induces an action on (infinite) rays in $\tm$.
Thus the inverted orbit $o,o.Y_1^{-1},\ldots
o.Y_n^{-1}$ can be viewed as a sequence of such  rays in the tree $\tm$. Their union is a
subtree of $\tm$ which we call the \textbf{ray tree}. Each ray
in the ray tree ends with an infinite sequence of zeros.
\begin{lemma}\label{l:raytree1}
If a vertex $v$ has exactly one child $wv$ in the ray tree then the section at $wv$ is just a propagating action, and the set of descendants of $wv$ in the ray tree is just the 0-ray from $wv$.
\end{lemma}
\begin{proof}
Let $|v|$ denote the distance of $v$ from $o$ in the graph metric.
Consider the section at $v$ of the word $Y_n$, written in alternating form $s_v = \sigma_1 a_1 \sigma_2 a_2 \sigma_3 a_3 .. \sigma_k$, where  $\sigma_i\in \text{Sym}(m_{|v|+1})$ and the $a_i$ are propagating actions in the descendant subtree of $v$.

Since $v$ has only one child $wv$ in the ray tree, we have
$$
w=0.\sigma_1^{-1} = 0.\sigma_2^{-1}\sigma_1^{-1} = ... = 0.\sigma_{k-1}^{-1}\cdots \sigma_1^{-1}
$$ This implies that  $0.\sigma_2 = 0.\sigma_3 = 0.\sigma_{k-1} =0$, so the section at $wv$ is just $a_1a_2a_3..a_{k-1}$, a propagating action. The set of descendants of $v$ in the ray tree is exactly the union of rays $(o.\sigma_1^{-1})v, (o.\sigma_2^{-1}\sigma_1^{-1})v,\dots, (o.\sigma_{k-1}^{-1}\cdots \sigma_1^{-1})v$, i.e. the ray $\cdots 0000wv$, as required.
\end{proof}

We can now show the following claim:
\begin{lemma}\label{l:number of ray trees}
The number of ray trees with with $r$ infinite rays is at most
$(m_*+1)^{6r}$.
\end{lemma}
\begin{proof}
Prune the ray tree at each ray after the first lone child. In other words, if $v$ is the first vertex in a ray to have only one child $u$ in the ray tree, then $u$ is in the pruned tree but its descendants are not.
Note that the pruned tree determines the ray tree: the ray tree is given by adding $0$-rays to the leaves of the pruned tree.

The pruned tree is finite;
let $i$ be the number of its vertices.
Since we pruned each ray at the first lone child, any vertex in the
pruned tree save for the leaves and their parents has at least two
children. If the total number of leaves is $r$, then the total number of children is at least $r+(i-2r)*2$, we get the inequality
$i\le 3r-1$.

The tree can be described by a path that visits every vertex and passes through every edge at most twice; such a path can be described by $2i$ numbers between $0$ and $m_*$ giving a list of directions which neighbor to go to next. The number of such lists is at most $(m_*+1)^{6r}$.
\end{proof}

\begin{lemma}\label{l:nr group elements per ray tree}
The number of group elements for which there are words with a given
ray tree with $r$ rays is at most $(m_*!)^{3m_*^2r}$.
\end{lemma}

\begin{proof}
For a group element $g$ its section at a vertex determines all the sections on the descendant of the vertex. Define the reduced section at a vertex $v$  as the permutation part of the section $a$ in its decomposition $\gamma=\langle \gamma_0,\ldots ,\gamma_{m-1} \rangle \sigma$.

A finite rooted subtree of $\tm$ is called full if each vertex either has maximal number of children or no child.

Given a rooted full subtree of $\tm$, the sections of $g$ at the leaves together with the reduced sections at the rest of the subtree determine the action of $g$ on $\tm$.

Now given a word, consider the pruned version $T'$ of its ray tree. Let $T''$ be the smallest full subtree of $\tm$ containing $T'$. This is called the minimal tree in \cite{BrieusselEntropy}.

By Lemma \ref{l:raytree1} the sections at the leaves of $T''$ are either permutations or propagating actions. The number of vertices of $T''$ is at most $m_*$ times the number of vertices in $T'$, which is at most $3r$ by Lemma \ref{l:number of ray trees}.

The total number of permutations together with propagating actions is bounded above by $m_*! + 2!3!\cdots m_*!\le (m_*!)^{m_*}$. So the total number of elements for which there is a word with a given ray tree is bounded above by $(m_*!)^{3m_*^2r}$.
\end{proof}

The next proposition contains an upper bound on the entropy of the walk on
$\mm$. The same bound with a multiplicative error term of $n^{o(1)}$ was first proved by \cite{BrieusselEntropy} using techniques from \cite{BKN}.

We take a different approach, describing an element $Y_n$ through the inverted orbit of the random walk word and the sections of $Y_n$.

This is similar to the approach in \cite{BrieusselEntropy}, where a slightly different tree is used to describe a word. The ray-tree description of the inverted orbit together with our tight estimates on
the expected size of the inverted orbit allow us to get a tighter
bound.

\begin{corollary}\label{c:entropies} The conditional entropy satisfies
\begin{eqnarray}
H(\supp Q_n\;|\;|Q_n|)&\le& 6\log(m_*+1)\ev |Q_n|, \notag
\\
H(\supp Q_n)&\le& 6\log(m_*+1)\ev |Q_n|+\log(n+1), \notag
\\
H(Y_n\;|\;|Q_n|)&\le& 5\ev|Q_n| m_*^3\log(m_*), \notag
\\
H(Y_n)&\le& 5  m_*^3\log(m_*)\ev|Q_n| + \log(n+1). \notag
\end{eqnarray}
\end{corollary}
\begin{proof}
The conditional entropy is the expected entropy of the conditional distribution. However, the ray-tree is determines the support of $Q_n$, and it has $|Q_n|$ infinite rays. Thus the conditional distribution of $\supp|Q_n|$ given $|Q_n|$ is supported on a set of size at most $(m_*+1)^{6|Q_n|}$. The entropy of the conditional distribution is bounded by the logarithm of this. The first claim follows by taking expectations, and the second follows by the trivial upper bound $H(|Q_n|)\le \log(n+1)$.

For the second claim, note that the number of group elements with a word with a given ray tree with $r$ leaves is at most $(m_*!)^{3m_*^2r}$ by Lemma \ref{l:nr group elements per ray tree}, so multiplying by the number of $r$-trees and taking logarithms gives the bound. We used the inequality $6\log(m_*+1)+3m_*^2\log(m_*!)\le 5 m_*^3\log m_*$.

For the last inequality, use the conditional entropy formula and the
fact  that $1\le |Q_n|\le n+1$, so $H(|Q_n|)\le \log(n+1)$.
\end{proof}

\begin{theorem}\label{t:speed m general}
\begin{eqnarray*}
 \ev|X_n| &\ge& \frac{1}{8000m_*^2}n^\alphan
 \underline \lambda\left(\frac{n^{1 - \alphan}}{500m_*^2}\right),
 \\
\ev |X_n| &\le& 6n^\alphan \overline \lambda(n^{1-\alphan})+ 48 m_*^2 n^{\frac{1+\alphan}2}.
\end{eqnarray*}
\end{theorem}

By \cite{LeePeres}, for an infinite group, $\overline \lambda(n)$ is at least a $c \sqrt{n}$. In particular, in the second inequality, the first term is dominant. Thus our bounds are tight up to constant factor, since by \eqref{e:monspeed} we have  $\overline \lambda\le c' \underline \lambda$.

\begin{proof}
Corollary \ref{c:Qn sum return} gives
$$
\ev |Q_n| = \sum_{i=0}^n \pr (T>i)
$$
Corollary \ref{c:return_probs_wreath} gives
\begin{eqnarray*}
\pr(T>n)\ge  \frac{1}{500 m_*^2}\,  n^{\alphan-1}, \qquad \qquad
\sum_{i=0}^{n}  \pr(T>i)\le
2n^\alphan.
\end{eqnarray*}
Corollary \ref{c:entropies} gives
\begin{eqnarray}
H(\supp Q_n)&\le& 6\log(m_*+1)\ev |Q_n|+\log(n+1), \notag
\\
H(Y_n)&\le& 5  m_*^3\log(m_*)\ev|Q_n| + \log(n+1). \notag
\end{eqnarray}
and finally,  Corollary \ref{c:spedd_in_pwp} gives that for any $p\le \pr(T>n)$ and any $q\ge \ev |Q_n|$ we have
\begin{eqnarray*}
\ev |X_n| &\ge& \frac{np\underline \lambda(1/p)}{16}, \\
\ev |X_n|&\le &3\overline \lambda\left(n/q\right)q+12\sqrt{n} \sqrt{H(Y_n) + H(\supp Q_n) + \ev |Q_n| }.\end{eqnarray*}
Combining these inequalities yields the Theorem.
\end{proof}

\begin{theorem}\label{t:entropy m general}
\begin{eqnarray*}
H(X_n) &\ge& \frac{n^{2\alphan-1}\;\underline \lambda^2(n^{1-\alphan})}{2^{30}m_*^4},
\\
H(X_n) &\le&  \left(15 m_*^4 + 2
h(n^{1-\alphan})\right)  n^\alphan.
\end{eqnarray*}
Denote $h_1$ the entropy of the step distribution. Another lower bound, useful for finite or small lamp groups is
\begin{eqnarray*}
H(X_n) &\ge& \frac{h_1}{500 m_*^2}\,n^{\alphan}.
\end{eqnarray*}

\end{theorem}
\begin{proof}
By Corollary \ref{c:spedd_in_pwp} for every $p\le \pr(T>n)$ and for every $q\ge \sum_{i=0}^n \pr (T>i)$ we have
\begin{eqnarray*}
H(X_n) &\ge& \frac{np^2\underline \lambda^2(1/p)}{4096},
\\
H(X_n) &\le&  H(Y_n) + H(\supp Q_n|\, |Q_n|) + 2
q\left(h\left(n/q\right) + \log (n+1)\right).
\end{eqnarray*}
And for finite $\Lambda$  by  (\ref{e:finitelamp}) we have
$$ H(X_n) \le H(Y_n) + q\log|\Lambda| + H(\supp Q_n). $$
Using the first upper bound for infinite $\Lambda$, and the second for finite $\Lambda$, the first two inequalities of the Theorem follow just as in Theorem \ref{t:speed m general}. For the last inequality, note that for $X_n=(\mathcal L_n,Y_n)$ we have
$$
H(X_n)\ge H(\mathcal L_n| Q_n) \ge h_1 \ev Q_n
$$
where we used the monotonicity of entropy.
\end{proof}

\begin{proof}[Proof of Theorems \ref{t:main speed} and \ref{t:main entropy}.]
Theorem \ref{t:main entropy} follows directly from Lemma \ref{l:choosing m} and Theorem \ref{t:entropy m general} for finite lamp groups, since $h(\cdot)$ in this case is bounded by a constant.

For Theorem \ref{t:main speed},
 we could use Theorem \ref{t:maingeneral}, but in this case there are more direct arguments. Note that given such $f$, $\tilde f(n)=f^2(n)/n$ satisfies the conditions of Lemma \ref{l:choosing m}. The theorem now follows from Theorem \ref{t:speed m general}, and the fact that simple random walk on $\mathbb Z$ satisfies
\[
\frac{2}{\pi} \sqrt{k}\leq \underline \lambda(k) = \overline \lambda(k)  \le \sqrt{k}. \qedhere
\]\end{proof}

\begin{remark}
By following the constants throughout the theorems, we get the following explicit  constants for the bounds speed and entropy of the groups constructed in the main theorems. Recalling that in lemma \ref{l:choosing m} one may take $m_*\leq C_\gamma = 3e^\frac{1}{1-\gamma}$.\\
For Theorem \ref{t:main entropy} we get
$$
 \frac{\log 2}{1000 C_\gamma^3}f(n) \leq H(X_n) \leq 16 C_\gamma^{9/2} f(n) \\
$$
and for  Theorem \ref{t:main speed}:
$$
 \frac{1}{2^{19} C_{2\gamma-1}^{7/2}}f(n) \leq E|X_n| \leq  49C_{2\gamma-1}^{9/4} f(n)
 $$
The constants can be improved by a more careful calculation. There are also other points in the proof were more careful choices would improve the constants, such  as  choosing a different set of generators for the base group and noticing the groups constructed actually have a degree sequence consisting of $2$ and $m_*$ only.
\end{remark}

\section{Open questions}

First, it would be nice to get the missing exponents; let us conjecture a positive answer to the question of Vershik and Peres:
\begin{conjecture} For every $\beta\in (1/2,3/4)$ there exists a random walk on a finitely generated group with
$$
\ev |X_n|\asymp n^\beta.
$$
\end{conjecture}
It is also interesting what exponents one can get for finitely presented groups, of which there are only countably many. We believe that our examples can be modified to get a dense set of exponents in $[3/4,1]$, but not to make the exponents oscillate. In this more restricted world, the following is still open.
\begin{question}
Let $X_n$ be a random walk on a finitely presented group. Is there necessarily a $\beta$ so that
$$
\ev |X_n|= n^{\beta +o(1)}?
$$
\end{question}
Related to this, we also have
\begin{conjecture}
The set of $\beta$ so that there exists a finitely presented group and a random walk on the group so that
$
\ev |X_n|= n^{\beta +o(1)}
$ is dense in $[1/2,1]$.
\end{conjecture}
The {\bf Hilbert compression exponent} $\alpha^*(G)$ of a group
of $G$ is the supremum over all
$\alpha\ge 0$ such that there exists a $1$-Lipschitz mapping $f:G\to
L_2$ and a constant $c>0$ such that for all $x,y\in G$ we have
$\|f(x)-f(y)\|_2\ge c\mbox{dist}(x,y)^\alpha.$ \citet*{ANP} showed that the speed exponent $\beta(G)$, if exists, satisfies
\begin{equation}\label{e:ANP}
\alpha(G)\beta(G)\le 2.
\end{equation} Thus every embedding gives an upper bound on the speed exponent. We do not know if this bound can be sharp for our examples $\mathcal M_m$ in the case when $\beta$ exists.
\begin{question} Assume that
$m$ is a constant sequence. Is it true that $\alpha(\mathbb Z\wr_S \mathcal M_m) \beta(Z\wr_S \mathbb M_m)$=2?
\end{question}
This would give an alternative upper bound, and apart from an  $n^{o(1)}$ fraction of error, would strengthen our theorem. It would also give new examples for Hilbert compression exponents.

The groups we are considering are permutation wreath products by piecewise automaton groups. It is still an open question to give the precise speed, or even entropy exponent on any of the classical automaton groups (except perhaps the lamplighter group); we will skip the definition of these groups but they can be found in, for example \cite{AAV}.
\begin{problem}
Specify the speed exponent or the entropy exponent in any of the following groups: Grigorchuk's group, Basilica group,  Hanoi towers group, degree-zero mother groups.
\end{problem}

\smallskip
\noindent{\bf Acknowledgments.} We thank Omer Angel for useful discussions, as well as Gady Kozma for the idea of Lemma \ref{l:Gady} which allowed us to remove a logarithmic error term from our bounds. We also think P\'eter Mester and Andrew Stewart for carefully reading previous versions. The work of the first author was supported by the Israeli Science Foundation. The research of G.A.\ was supported by the Israel Science Foundation grant ISF $1471/11$ and by Young Scientist grant $\# I-1121-304.1-2012$ from the German Israeli Foundation. The work of
B.V.\ was supported by the Canada Research Chair program, the Marie Curie grant SPECTRA, the OTKA grant $K109684$, and the NSERC Discovery Accelerator Program.
\bibliographystyle{dcu}
\bibliography{speedexponent}

\bigskip\bigskip\bigskip\noindent
\begin{minipage}{0.49\linewidth}
Gideon Amir
\\Department of Mathematics
\\Bar Ilan University
\\Ramat Gan, 52900 Israel
\\{\tt amirgi@macs.biu.ac.il }
\end{minipage}
\begin{minipage}{0.49\linewidth}
B\'alint Vir\'ag
\\Departments of Mathematics and Statistics
\\University of Toronto
\\Toronto ON~~M5S 2E4, Canada
\\{\tt balint@math.toronto.edu}
\end{minipage}

\end{document}